\documentstyle[11pt,amssymb,amsthm,amsmath,multirow,verbatim]{article}
\newtheorem{thm}{Theorem}
\newtheorem{lm}[thm]{Lemma}
\newtheorem{prop}[thm]{Proposition}
\newtheorem{cor}[thm]{Corollary}
\newtheorem{deff}[thm]{Definition}

\newtheorem{ex}[thm]{Example}

\newtheorem{obs}[thm]{Observation}

\newcommand{\Comment}[1]{}

\begin{document}

\title{ 
 Constructing commutative semifields of square order}

\author{   S. M. Gagola III$^*$ and
         J. L. Hall$^\dag$ \\
             Department of Algebra$^{*\dag}$ \\
             Charles University \\
             Prague 8, Czech Republic \\
              gagola@karlin.mff.cuni.cz      \\
\\
Mathematical Sciences School$^\dag$\\
Queensland University of Technology\\
Brisbane, QLD Australia\\
j42.hall@qut.edu.au
 }

\date{}

\maketitle

\begin{abstract}
The projection construction has been used to construct semifields of odd characteristic using a field and a twisted semifield [Commutative semifields from projection mappings, \it Designs, Codes and Cryptography, \bf 61 \rm (2011), 187--196].  We  generalize  this idea to a projection construction using two twisted semifields to construct semifields of odd characteristic.  Planar functions and semifields have a strong connection  so this also constructs new planar functions.

\end{abstract}
\vspace{-.1 cm}

\vspace{-.4cm}
\section{Introduction}
Semifields are algebraic structures satisfying most of the axioms of a field.  The classification of finite fields has been concluded over a century ago, however the classification of finite semifields is far from complete. 
Semifields have connections with geometry \cite{CH2008,Knuth1965}, and a  connection with planar functions means that semifields have applications in  classical cryptographic systems \cite{CDR98}, quantum cryptographic systems \cite{RS07}, wireless communication  \cite{DY2007},  and coding theory \cite{Horadam07}.

 Commutative semifields with odd characteristic are equivalent to those planar functions that are known as Dembowski-Ostrom Polynomials  (DO polynomials) \cite{CH2008}.
Planar functions belong to the larger class of highly nonlinear functions which are of use in the above mentioned applications  as well as being of theoretical interest \cite{CD04,CM97D}.
  A look at a recent list of planar functions  and semifields \cite{BH2011,ZhouPott2013} shows no obvious pattern. Computational searches can discover new planar functions and semifields \cite{JZH2012}, but algebraic work is required to significantly deepen our understanding.

The motivation behind this work is to generalize the projection construction for planar functions of Bierbrauer \cite{Bier2011}.  We use the trace map with Dembowski-Ostrom polynomials which has also been used to explore Dembowski-Ostrom polynomials that are permutations \cite{BCHO2011}.


Let $p$ be an odd prime.  Suppose that $S_g$ and $S_h$ are semifields of order $p^{2r}$ associated with planar functions $g$ and $h$ respectively.  Here $S_g$ and $S_h$ are isotopic to presemifields with multiplication operations 
\begin{alignat*}{5}
&&&x \circ y = g(x+y)-g(x)-g(y) && \\
&\textnormal{and \quad} &&x \triangle y = h(x+y)-h(x)-h(y) &&\textnormal{\quad respectively.}
\end{alignat*}
Now we consider a new operation
\begin{equation}\label{eqn:newstar}
x*y=x\circ y+\overline{(x\circ y) }+x\triangle y-\overline{(x\triangle y)}
\end{equation}
where $\bar{x}=x^{p^r}$.  
Note that 
$$ x*y=f(x+y)-f(x)-f(y) $$
where
\begin{equation}
 f(x)=g(x)+g(x)^{p^r} + h(x) - h(x)^{p^r}.   \label{eqn:newplanar}
\end{equation}
Thus, if $f(x)$ is also planar over ${\mathbb F}_{p^{2r}}$, then the multiplication operation in Equation~(\ref{eqn:newstar}) forms another presemifield, say the {\em projection construction}. 

Equation (\ref{eqn:newplanar})  has been used to construct a  planar function (and hence a semifield)  by Bierbrauer \cite{Bier2011}.  The contrast with the current work is that in \cite{Bier2011} 
the commutative semifield is associated with a planar function of the form (\ref{eqn:newplanar}) where $g(x)=x^2$.
Whereas, in this paper, we allow both  $g(x)$ and $h(x)$ to be a wider variety of polynomials.
We investigate the conditions on the polynomials $g(x)$ and $h(x)$ that result in the  function $f(x)$ being  planar over ${\mathbb{F}}_{p^{2r}}$.
Since we are primarily interested in constructing semifields,   we have restricted our search for planar functions of the shape of \mbox{Equation (\ref{eqn:newplanar})}, to DO polynomials.

We begin  by showing that if $g(x)$ and $h(x)$ are planar Dembowski-Ostrom monomials over ${\mathbb{F}}_{p^{2r}}$ with either $g(x)=x^2$ or $h(x)=x^2$, then 
$f(x)$, as defined in Equation~(\ref{eqn:newplanar}), is also planar over ${\mathbb{F}}_{p^{2r}}$.

In Section \ref{sec:known},  we show that if one of the semifields $S_g$ or $S_h$ is a field and the other is a commutative semifield, then the semifield obtained using the projection construction (as in Equation (\ref{eqn:newstar})) is already known.  The planar functions that these   correspond to  have either $g(x)=x^2$ or $h(x)=x^2$, where the other, namely $g(x)$ or $h(x)$, is another planar DO polynomial.  Several known families of semifields are shown to also fit the projection construction.

Section \ref{sec:nuclei} presents results on determining the middle and left nuclei of  commutative semifields constructed using this projection construction.

In Section \ref{sec:new}, some computationally derived semifields/planar functions are presented.  Computations of the sizes of the middle nuclei are used to show that these commutative semifields are new and not isotopic to previously known semifields.



\section{Preliminaries \label{sec:pre}}

\begin{deff}
A \it semifield \rm ${\mathbb{S}}$ is an algebraic structure with two binary operations, addition, $+$,  and multiplication, $*$, such that
\begin{itemize}
\item $({\mathbb{S}},+)$ is an abelian group,
\item $({\mathbb{S}}\diagdown \{0\},*)$ is a loop, and
\item multiplication is distributive on both the left and right.
\end{itemize}
\end{deff}

A \it presemifield \rm ${\mathbb{P}}$ is similar to a semifield, but with one of the axioms relaxed.  Namely, $({\mathbb{P}}\diagdown \{0\},*)$ does not necessarily have a multiplicative identity and is therefore a quasigroup.  A
presemifield is \it commutative \rm if its multiplication is commutative.   Semifields can  be non-commutative and non-associative but, by Wedderburn's Theorem \cite{WT}, if a semifield is finite, then associativity implies commutativity.   Given a commutative presemifield with multiplication $x * y$ it is easy to construct a commutative semifield.

\begin{lm}\cite{Albert43}
Every quasigroup $(Q,*)$ is isotopic to a loop.
For $a,b\in Q$, $(Q,*)$ is isotopic to $(Q,\circ)$ where
\[ x\circ y=(x)R_b^{-1}*(y)L_a^{-1} \]
\[ \Longleftrightarrow \:\:(x*b)\circ (a*y)=x*y. \qquad\quad {}^{}\]
Here $(a*b)\circ x = x = x\circ (a*b)$ for any $x\in Q$.
\end{lm}

Two presemifields ${\mathbb{S}}_1=({\mathbb{S}},+,*)$ and ${\mathbb{S}}_2=({\mathbb{S}},+,\circ)$ are called isotopic if there exist three linearized permutation polynomials $L_1,L_2,L_3$ over ${\mathbb{S}}$ such that $L_1(x) \circ L_2(y) = L_3(x * y)$ for any $x, y \in{\mathbb{S}}$.  

Let ${\mathbb{S}}$ be a semifield. The subsets
\begin{alignat*}{1}
N_{l}({\mathbb{S}}) &= \left\{a \in {\mathbb{S}}\:|\: a (xy) = (ax) y \textnormal{ for all } x,y\in{\mathbb{S}} \right\},   \\
N_{m}({\mathbb{S}}) &= \left\{a \in {\mathbb{S}}\:|\: x (ay) = (xa) y \textnormal{ for all } x,y\in{\mathbb{S}}\right\}, \\
N_{r}({\mathbb{S}}) &= \left\{a \in {\mathbb{S}}\:|\: x (ya) = (xy) a \textnormal{ for all } x,y\in{\mathbb{S}}\right\},
\end{alignat*}
are called the \it left, middle \rm and \it right nucleus \rm of ${\mathbb{S}}$, respectively. These sets are finite fields \cite{Knuth1965}.
The intersection $N({\mathbb{S}}) = N_{l}({\mathbb{S}})\cap N_{m}({\mathbb{S}})\cap N_{r}({\mathbb{S}})$ is called the \it nucleus \rm of ${\mathbb{S}}$. In the case where ${\mathbb{S}}$ is commutative it can easily be shown that $N({\mathbb{S}}) = N_{l}({\mathbb{S}}) =  N_{r}({\mathbb{S}}) \subseteq N_{m}({\mathbb{S}})$.  The size of the nuclei is invariant under isotopism, a fact which is used in Section \ref{sec:new} to show non isotopism of some semifields.

We will be using polynomials and on finite fields to explore semifields.
Let \[D_{f}(x,a)=f(x+a)-f(x)-f(a).\]

\begin{deff}\rm Let ${\mathbb{F}}_{p^r}$ be a field of characteristic $p$. A function $f : {\mathbb{F}}_{p^r} \rightarrow {\mathbb{F}}_{p^r}$ is  called a \it planar function \rm if for every  $a \in {\mathbb{F}}_{p^r}^*$ the 
map $x \mapsto D_{f}(x,a)$
is a bijection.
\end{deff}
Note that the definition of planar function does not require $p$ to be odd, however planar functions cannot exist on fields of even characteristic \cite{DO1968}, thus the focus of this work is on polynomials on fields of odd characteristic.

A polynomial $f(x)\in{\mathbb{F}}_{p^r}[x]$ is a \emph{Dembowski-Ostrom} polynomial 
if its reduced form  has the shape
\begin{equation}
f(x)=\sum_{i,j=0}^ka_{ij}x^{p^i+p^j}. \nonumber
\end{equation}
 Any polynomial $f(x)\in{\mathbb{F}}_{p^r}[x]$ may be \emph{reduced} modulo $x^{p^r}-x$, which yields a polynomial function of degree less than $p^r$ that induces the same function on ${\mathbb{F}}_{p^r}$.  A polynomial which is planar over ${\mathbb{F}}_{p^r}$ and also planar when restricted to the subfield ${\mathbb{F}}_p$ is planar over every subfield.


\begin{lm}\cite{CH2008}\label{lm:Df}
Let $f(x)$ be a DO polynomial on ${\mathbb{F}}_{p^r}$.  Then $f(x)$ is planar if and only if $D_f(x,y)\neq 0$ for all $x,y\in{\mathbb{F}}_{p^r}^*$.
\end{lm}

\begin{thm}\cite{CH2008}\label{thm:planarsemifield}
Let $p$ be an odd prime.  Each equivalence class of  DO planar polynomials on ${\mathbb{F}}_{p^r}$ constructs a unique isotopy class of  commutative semifields of order $p^r$.  Furthermore all commutative semifields of order $p^r$ can be constructed from a planar DO polynomial on ${\mathbb{F}}_{p^r}$.
\end{thm}
Theorem \ref{thm:planarsemifield} shows that  commutative semifields can be investigated by investigating planar DO polynomials.  Notions of equivalence of polynomials are of importance in applications \cite{CCZ1998}.   Planar polynomials $\Pi$ and $\Pi'$ with
\[\Pi'(x)=\left(L_1\circ\Pi\circ L_2\right)(x)+L_3(x)\]
are said to be \emph{EA-equivalent} if $L_1(x)$, $L_2(x)$, $L_3(x)$ are affine polynomials and $L_1(x)$, $L_2(x)$ are bijections \cite{KP2008}; and \emph{linear equivalent} if $L_1(x)$, $L_2(x)$, $L_3(x)$ are linearized polynomials and $L_1(x), L_2(x)$ are bijections.  CCZ equivalence \cite{CCZ1998} is of interest in cryptographic applications.  For planar  Dembowski-Ostrom polynomials CCZ, EA and linear equivalence are equivalent \cite{BH2011C}.


  It has been shown that linear equivalent planar functions construct isotopic semifields \cite{CH2008}.
The proof of Theorem~\ref{thm:planarsemifield} is constructive: let $x,y\in{\mathbb{F}}_q$ and  $*$ be multiplication on a presemifield of order $q$, then there exists a planar $DO$ polynomial $f$ such that
\[x*'y= D_f(x,y)\]
forms a presemifield isotopic to $({\mathbb{F}}_q,+,*)$.
So in  showing the non-equivalence of planar DO polynomials it is sufficient to show that their corresponding semifields are non-isotopic.


There are several known planar functions that fit the shape of Equation~(\ref{eqn:newplanar}) for which $g(x)$ and $h(x)$ are different (see Section \ref{sec:known}).
The following theorem establishes necessary and sufficient conditions on $g(x)$ and  $h(x)$ to form a planar function using \mbox{Equation (\ref{eqn:newplanar})}.

\begin{thm}\label{thm:newplanar}
Suppose $p$ is an odd prime.
Let $g(x)$ and $h(x)$ be functions on ${\mathbb{F}}_{p^{2r}}$ and let
$$ f(x)=g(x)+(g(x))^{p^r}+h(x)-(h(x))^{p^r}, $$
then the polynomial $f(x)$ is planar if and only if
for any 
$a,b\in {\mathbb{F}}_{p^{2r}}^*$
either
\[D_{g}(a,b) + \left(D_{g}(a,b)\right)^{p^r}\neq 0 \quad\textrm{or}\quad 
D_{h}(a,b)\not\in {\mathbb{F}}_{p^r}.\]
\end{thm}

\begin{proof}
Since ${\mathbb{F}}_{p^r}$ is a subfield of ${\mathbb{F}}_{p^{2r}}$,
the trace function
\[Tr(x)=Tr_{{\mathbb{F}}_{p^{2r}}/{\mathbb{F}}_{p^{r}}}(x)=x+x^{p^r}\]
can be used in
rewriting $f(x)$:
\[f(x)=Tr\left(g(x)\right)+h(x)-(h(x))^{p^r}.\]
By Lemma \ref{lm:Df}, the polynomial $f(x)$ is planar if and only if $D_{f}(a,b)\neq 0$ for all $a,b\in{\mathbb{F}}_{p^{2r}}^*$.

Now assume that $D_{f}(a,b)= 0$. Then
\begin{align}
&&Tr\left(D_{g}(a,b)\right)+D_{h}(a,b)-\left(D_{h}(a,b)\right)^{p^r}=\:& 0\label{eqn:=01}\\
 &\Longrightarrow & \Big(Tr\left(D_{g}(a,b)\right)+D_{h}(a,b)-\left(D_{h}(a,b)\right)^{p^r}\Big)^{p^r}=\:& 0 \qquad\qquad\\
&\Longrightarrow & Tr\left(D_{g}(a,b)\right)+\left(D_{h}(a,b)\right)^{p^r}-D_{h}(a,b) =\:& 0.\label{eqn:=02}
\end{align}
By adding Equations (\ref{eqn:=01}) and (\ref{eqn:=02}) it follows that $D_{f}(a,b)= 0$ if and only if
\begin{align}
Tr\left(D_{g}(a,b)\right) & =0 \quad\quad \mbox{and}\label{eqn:05}\\
D_{h}(a,b)-\left(D_{h}(a,b)\right)^{p^r} & =0. \label{eqn:insubfield}
\end{align}
Thus, $f(x)$ is planar if and only if, for any $a,b\in{\mathbb{F}}_{p^{2r}}^*$, either Equation~(\ref{eqn:05}) or Equation~(\ref{eqn:insubfield}) does not hold.
\end{proof}




\section{\!Known planar functions that fit the projection construction \label{sec:known}}
\vspace{-.2cm}
The LMPTB functions \cite{Bier2011}  fit the shape of the projection construction; they are easily recognizable as being of the form of \mbox{Equation (\ref{eqn:newplanar})}.
\begin{thm}\cite[Thm 1]{Bier2011}\label{thm:Bier}
Let $g(x)=\frac{1}{2}x^2$ and
\vspace{-.1cm}
\begin{equation}
h(x)=\frac{1}{2}\sum^{k}_{i=0}(-1)^ix^{(1+p^2)p^{2i}}+\frac{1}{2}\sum_{j=0}^{k-1}(-1)^{k+j}x^{(1+p^2)p^{2j+1}}, \label{eqn:LMPT}
\end{equation}
\vspace{-.1cm
}then the function $f(x)$ as defined in \mbox{Equation (\ref{eqn:newplanar})} is a planar function on ${\mathbb{F}}_{p^{2(2k+1)}}$ for $k>0$.
\end{thm}
The LMPTB semifields are an example of the projection construction as described in  \cite{Bier2011}.  The projection construction can be used to construct a new semifield from
two or more semifields.  The polynomials that are used to construct a planar polynomial in Equation (\ref{eqn:newplanar}) are not always planar polynomials (See Example \ref{e:exam}).  Hence, the semifields that are derived from the planar polynomials of the shape of Equation (\ref{eqn:newplanar}) may not be contructable from other semifields.

The LMPTB semifields are isotopic to the Budaghyan-Helleseth semifields \cite{Marino}.  Under certain conditions \mbox{Equation (\ref{eqn:newplanar})} contains generalized Budaghyan-Helleseth polynomials with $g(x)=x^2$, and under other conditions,  contains planar functions with $h(x)=x^2$ which 
have corresponding semifields that are isotopic to commutative semifields constructed by Zhou  \& Pott \cite{ZhouPott2013}.

\begin{thm}\label{thm:x2}
Let $g(x)$ and $h(x)$ be planar Dembowski-Ostrom monomials over ${\mathbb{F}}_{p^{2r}}$.  If either $g(x)=x^2$ or $h(x)=x^2$, then
\[ f(x)=g(x)+g(x)^{p^r} + h(x) - h(x)^{p^r} \]
is planar over ${\mathbb{F}}_{p^{2r}}$.
\end{thm}

The following lemma 
will be used to prove Theorem \ref{thm:x2}.
\begin{lm}\cite[Thm 3.3]{CM97D} \label{lem:equiv}
Let $h(x)=x^{p^i+p^k}\in{\mathbb{F}}_{p^{n}}[x]$ with $0\leq i< k$, then the following are equivalent:
\begin{alignat*}{3}
&(i) & & h(x) \textrm{ is planar;} \\
&(ii) &&D_h(x,a)\neq 0 {\textrm{ for all }} x,a\in {\mathbb{F}}_{p^{n}}^*; \\
&(iii)\quad &&\left|\langle b^{p^k-p^i}\rangle\right| {\textrm{ is odd for any }} b\in {\mathbb{F}}_{p^{n}}^*; \\
&(iv) & & n/gcd(n,i-k) \textrm{ is odd}.
\end{alignat*}
\end{lm}

\begin{proof}[Proof of Theorem \ref{thm:x2}]

\rm Assume that $f(x)$ is not planar. Then there exist elements $x,a\in {\mathbb{F}}_{p^{2r}}^*$ such that $D_{f}(x,a)= 0$. Thus, from Theorem~\ref{thm:newplanar}, there exist $x,a\in{\mathbb{F}}_{p^{2r}}^*$ such that  $Tr\left(D_{g}(x,a)\right)=0$ and $D_h(x,a)\in {\mathbb{F}}_{p^r}$. Thus,
\begin{equation} \label{eqn:trg}
D_{g}(x,a)\left(1+D_{g}(x,a)^{p^r-1}\right)=D_{g}(x,a)+D_{g}(x,a)^{p^r}=0.
\end{equation}
By Lemma \ref{lm:Df},  $D_{g}(x,a)\neq 0$, hence  Equation (\ref{eqn:trg}) implies that $D_{g}(x,a)^{p^r-1}=-1$.

Let $G$ be the unique subgroup of ${\mathbb{F}}_{p^{2r}}$ 
such that $|G|=2\left|{\mathbb{F}}_{p^r}^*\right|=2\left(p^r-1\right)$.  Note that
\[ G=\left\{x\in{\mathbb{F}}_{p^{2r}}\:\left|\: x^{p^r-1}=\pm 1\right.\right\}.\]
Since $D_{g}(x,a)^{p^r-1}=-1$, $D_{g}(x,a)\in G\diagdown {\mathbb{F}}_{p^r}^*$.

Now let  $c=xa$ 
and note that
\begin{alignat}{1}
x^{p^i}a^{p^k}+x^{p^k}a^{p^i} &= x^{p^i}a^{p^i}\left(a^{p^{k}-p^i}+x^{p^{k}-p^i}\right), \nonumber \\
   &= c^{p^i}\left(c^{p^{k}-p^i}x^{p^{i}-p^k}+x^{p^{k}-p^i}\right). \label{e:relation}
\end{alignat}

\noindent\it Case 1. \rm Suppose $g(x)=x^{p^i+p^k}$ and $h(x)=x^2$.  Since $D_{g}(x,a)=x^{p^i}a^{p^k}+x^{p^k}a^{p^i}\in G\diagdown {\mathbb{F}}_{p^r}^*$ and $D_{h}(x,a)=2ax=2c\in {\mathbb{F}}_{p^r}$, by Equation (\ref{e:relation}),  $c^{p^{k}-p^i}x^{p^{i}-p^k}+x^{p^{k}-p^i}$ is contained in $G\diagdown {\mathbb{F}}_{p^r}^*$.

\noindent\it Case 2. \rm Suppose $g(x)=x^2$ and $h(x)=x^{p^i+p^k}$.  Since $D_{g}(x,a)=2ax=2c\in G\diagdown {\mathbb{F}}_{p^r}^*$ and $D_{h}(x,a)= x^{p^i}a^{p^k}+x^{p^k}a^{p^i}\in {\mathbb{F}}_{p^r}$, by Equation (\ref{e:relation}),  $c^{p^{k}-p^i}x^{p^{i}-p^k}+x^{p^{k}-p^i} \in G\diagdown {\mathbb{F}}_{p^r}^*$.

Therefore, in either case the sum $(c^{p^{k}-p^i}x^{p^{i}-p^k}+x^{p^{k}-p^i})$ is contained in $G\diagdown {\mathbb{F}}_{p^r}^*$.  By Lemma \ref{lem:equiv}, both $c^{p^{k}-p^i}$ and $x^{p^{k}-p^i}$ are of odd order.
Since ${\mathbb{F}}_{p^r}^* \leq G$ with $|G|=2|{\mathbb{F}}_{p^r}^*|$, we know that $c^{p^{k}-p^i},x^{p^{k}-p^i} \in {\mathbb{F}}_{p^r}^*$. Therefore, $(c^{p^{k}-p^i}x^{p^{i}-p^k}+x^{p^{k}-p^i})\in {\mathbb{F}}_{p^r}^*$
forming a contradiction.  Hence, there do not exist $x,a\in {\mathbb{F}}_{p^{2r}}^*$ such that \mbox{$D_{f}(x,a)=0$} and therefore, $f(x)$ is a planar function.
\end{proof}

We now show that all of the polynomials described in  Theorem \ref{thm:x2} are known.

\begin{thm}\label{thm:Pott}
Let $h(x)=x^2\in{\mathbb{F}}_{p^{2r}}[x]$, $g(x)=x^{p^k+1}$ be a planar function on ${\mathbb{F}}_{p^{2r}}$,  and $f(x)$ be of the form (\ref{eqn:newplanar}), then the semifield associated with $f(x)$ is isotopic to the Zhou-Pott commutative semifield, ${\mathbb S}_{k,id}$  \cite{ZhouPott2013}.
\end{thm}
\begin{proof}
For $f(x)=x^{p^k+1}+x^{(p^k+1)p^r}+x^2-x^{2p^r}\in {\mathbb{F}}_{p^{2r}}[x]$, the binary operation on the corresponding presemifield is
\begin{alignat*}{1}
x*y 
    &= x^{p^k}y+xy^{p^k}+\left(x^{p^k}y+xy^{p^k}\right)^{p^r}+2xy-2x^{p^r}y^{p^r}.
\end{alignat*}

Let $\omega \in {\mathbb{F}}_{p^{2r}}^*$ such that $\omega^{p^r}=-\omega$ and denote any $x\in {\mathbb{F}}_{p^{2r}}$ by $x=x_1+x_2\omega$ for $x_1,x_2\in {\mathbb{F}}_{p^{r}}$. Then, since $x^{p^r}=(x_1+x_2\omega)^{p^r}=x_1-x_2\omega$,
\begin{alignat*}{1}
x*y &= (x_1+x_2\omega)*(y_1+y_2\omega), \\
    &= 2x_1^{p^k}y_1+2x_1y_1^{p^k}+2(x_2\omega)^{p^k}(y_2\omega)+2(x_2\omega)(y_2\omega)^{p^k} \\
    &\quad+4x_1(y_2\omega)+4(x_2\omega)y_1.
\end{alignat*}

By \cite[Thm 1]{ZhouPott2013}, the Zhou-Pott presemifield ${\mathbb{P}}_{k,\sigma}=({\mathbb{F}}_{p^{2r}},+,\circ)$ has multiplication
\begin{alignat}{1}\label{eqn:12}
x\circ y &= (x_1+x_2\omega)\circ (y_1+y_2\omega) ,\nonumber\\
    &= x_1^{p^k}y_1+x_1y_1^{p^k}+\alpha \left(x_2^{p^k}y_2+x_2y_2^{p^k}\right)^{\sigma}+(x_1y_2+x_2y_1)\omega ,
\end{alignat}
where $\alpha$ is a non-square element in ${\mathbb{F}}_{p^r}$ and $\sigma$ is a field automorphism of ${\mathbb{F}}_{p^r}$.
Let $\alpha=\frac{\omega^{p^k+1}}{4}$. Since $\frac{2r}{{\textnormal{gcd}}(2r,k)}$ is odd, $k$ is even, and therefore, $p^k+1$ is not divisible by four.  Since $\omega^{p^r}=-\omega$ and $\frac{p^k+1}{2}$ is odd, $\omega^{\frac{p^k+1}{2}}\notin {\mathbb{F}}_{p^{r}}$.  Hence, $\omega^{p^k+1}$ is a non-square element of ${\mathbb{F}}_{p^{r}}$, hence $\alpha$ is non-square.
  Let $\sigma$ be the identity map, then Equation (\ref{eqn:12}) becomes
\begin{alignat*}{1}
(x_1+x_2\omega)\circ (y_1+y_2\omega) &= x_1^{p^k}y_1+x_1y_1^{p^k}+ {\textstyle{\frac{1}{4}}}(x_2\omega)^{p^k}(y_2\omega)+{\textstyle{\frac{1}{4}}}(x_2\omega)(y_2\omega)^{p^k}  \\
    &\quad+x_1(y_2\omega)+(x_2\omega)y_1.
\end{alignat*}

Let $L_1$ and $L_2$ be linearized permutation polynomials over ${\mathbb{F}}_{p^{2r}}$ where $L_1(x)=3x-x^{p^r}$ and $L_2(x)={\textstyle{\frac{3}{2}}}x-{\textstyle{\frac{1}{2}}}x^{p^r}$. To show that $L_1$ is a permutation polynomial it is enough to show that the only solution to $3x-x^{p^r}=0$, is $x=0$.  If $3x=x^{p^r}$, then $x^{p^r-1}=\left(x^{p^r-1}\right)^{p^r}$, and thus, $x^{p^r-1}\in{\mathbb{F}}_{p^{r}}$.  Since $x^{p^r-1},x^{p^r+1}\in{\mathbb{F}}_{p^{r}}$, then $x^2$ is also contained in ${\mathbb{F}}_{p^{r}}$.  By squaring both sides of
$3x=x^{p^r}$, it follows that $9x^2=x^2$. Hence, $8x^2=0$ and $x=0$.\

Since
\begin{alignat*}{1}
L_1(x_1+x_2\omega) &\circ L_2(y_1+y_2\omega) = (2x_1+4x_2\omega)\circ(y_1+2y_2\omega) ,\\
    &= 2x_1^{p^k}y_1+2x_1y_1^{p^k}+ {\textstyle{\frac{8}{4}}}(x_2\omega)^{p^k}(y_2\omega)+{\textstyle{\frac{8}{4}}}(x_2\omega)(y_2\omega)^{p^k} \\
    &\quad+4x_1(y_2\omega)+4(x_2\omega)y_1 , \\
    &=  (x_1+x_2\omega)*(y_1+y_2\omega),
\end{alignat*}
the presemifield associated with $f(x)\!=\!x^{p^k+1}\!+x^{(p^k+1)p^r}\!+x^2-x^{2p^r}$ is \mbox{isotopic to ${\mathbb{S}}_{k,id}$}.
\end{proof}

The generalized Budaghyan-Helleseth functions  require some rearranging to show they can be constructed using \mbox{Equation (\ref{eqn:newplanar})}.
From \cite[Thm 3]{Bier2011},
\begin{equation}\label{eqn:GBH}
B(x)=x^{p^r+1}+\omega\beta x^{p^s+1}+\omega\beta^{p^r}x^{(p^s+1)p^r}
\end{equation}
where $Tr(\omega)=0$, $x^{p^s}\neq -x$ for all $x\in{\mathbb{F}}_{p^{2r}}^*$ and $\beta^{p^r-1}$ is not in the subgroup of order $(p^r+1)/{\textnormal{gcd}}(p^r+1,p^s+1)$.  Their corresponding presemifields are usually denoted by $BH(p,m,s,\beta)$.

\begin{obs}
Note that the generalized Budaghyan-Helleseth functions can be written as
\[B(x)=x^{p^r+1}+\omega\beta x^{p^s+1}-\left(\omega\beta x^{(p^s+1)}\right)^{p^r}\]
since $Tr(\omega)=0$.  Thus, $B(x)+B(x)^{p^r}=2x^{p^r+1}$.
\end{obs}

\begin{thm}\label{thm:BHfunc}
Let $g(x)=x^2\in{\mathbb{F}}_{p^{2r}}[x]$, $h(x)=x^{p^k+p^i}$ be a planar function on ${\mathbb{F}}_{p^{2r}}$ and $f(x)$ be of the form (\ref{eqn:newplanar}). If either $r$ is even or $p\equiv 1$ $($mod $4)$, then $f(x)$ is equivalent to a generalized Budaghyan-Helleseth function.
\end{thm}
\begin{proof}
Note that since either $r$ is even or $p\equiv 1$ $($mod $4)$, there exists an element $j \in {\mathbb{F}}_{p^r}$ such that $j^2=-1$.  Also note that since $k-i$ is even, $j^{p^k+p^i}=-1$ and $j^{p^k}=j^{p^i}$. Let $L_1(x)=c_0x+jc_0^{p^r}x^{p^r}$  with $c_0\in {\mathbb{F}}_{p^{2r}}^*$. Then
\[ f(L_1(x))=4jc_0^{p^r+1}x^{p^r+1}+2c_0^{p^k+p^i}x^{p^k+p^i}-\left(2c_0^{p^k+p^i}x^{p^k+p^i}\right)^{p^r}\!\!\!. \]

\noindent\it Case 1. \rm Suppose $i=0$ and $L_2(x)=\frac{1}{4jc_0^{p^r+1}}x$.
Then \[ L_2(f(L_1(x)))=x^{p^r+1}+\textstyle{\frac{c_0^{p^k-p^r}}{2j}x^{p^k+1}-\left(\frac{c_0^{p^k-p^r}}{2j}x^{p^k+1}\right)^{p^r}}\!\!.\]

\noindent\it Case 2. \rm Suppose $i=r$ and $L_2(x)=\textstyle{\frac{1}{4jc_0^{p^r+1}}x^{p^r}}\!\!\!.$
Then \[ \textstyle{L_2(f(L_1(x)))=x^{p^r+1}+\frac{c_0^{p^{k+r}-p^r}}{2j}x^{p^{k+r}+1}-\left(\frac{c_0^{p^{k+r}-p^r}}{2j}x^{p^{k+r}+1}\right)^{p^r}\!\!\!.} \]

\noindent\it Case 3. \rm Now suppose that $i\notin\{0,r\}$ and $L_2(x)=\frac{1}{8jc_0^{p^r+1}}x+\frac{1}{8jc_0^{p^r+1}}x^{p^r}+d_{2r-i}x^{p^{2r-i}}-d_{2r-i}x^{p^{r-i}}$ where $d_{2r-i}\in {\mathbb{F}}_{p^r}$.
Then \[ L_2(f(L_1(x)))=x^{p^r+1}+4c_0^{p^{k-i}+1}d_{2r-i}x^{p^{k-i}+1}-\left(4c_0^{p^{k-i}+1}d_{2r-i}x^{p^{k-i}+1}\right)^{p^r}\!\!\!.\]
\end{proof}
Note that if $r$ is odd and $p\equiv 1$ $($mod $4)$, then the Budaghyan-Helleseth semifields  belong to the class of semifields ${\mathbb S}_k$ as defined in  Theorem~\ref{thm:Pott} \cite{ZhouPott2013}.

\begin{thm}\label{thm:Pott2}
Let $g(x)=x^2\in{\mathbb{F}}_{p^{2r}}[x]$, $h(x)=x^{p^k+1}$ a planar function on ${\mathbb{F}}_{p^{2r}}$ and $f(x)$ be of the form (\ref{eqn:newplanar}). If $p\equiv 3$ $($mod $4)$ and $r$ is odd, then the semifield associated with $f(x)$ is isotopic to a Zhou-Pott semifield, ${\mathbb{S}}_{k,id}$.
\end{thm}

\begin{proof}
Let $\beta\in {\mathbb{F}}_{p^{2r}}^*$ such that $\beta^4=-1$.  Since $k$ is even and $p\equiv 3$ (mod 4), $p^k+1\equiv$ 2 (mod 8). Thus, $\beta^{p^k+1}=\beta^2$.  Furthermore, since 
$\beta^2$ has order four, $\left(\beta^2\right)^{p^r}=-\beta^2$. Therefore,
\begin{alignat*}{1}
&(\beta x)^2+(\beta x)^{2p^r}+(\beta x)^{p^k+1}-(\beta x)^{(p^k+1)p^r} \\
=\: &\beta^2 x^2+(\beta^2)^{p^r} x^{p^r}+\beta^2 x^{p^k+1}-(\beta^2)^{p^r} x^{(p^k+1)p^r} \\
=\: &\beta^2 x^2-\beta^2 x^{p^r}+\beta^2 x^{p^k+1}+\beta^2 x^{(p^k+1)p^r} \\
=\: &\beta^2 \left(x^{p^k+1}+x^{(p^k+1)p^r}+x^2-x^{p^r}\right)
\end{alignat*}
and $f(x)$ is equivalent to $x^{p^k+1}+x^{(p^k+1)p^r}+x^2-x^{2p^r}$.  Hence, by Theorem \ref{thm:Pott}, the semifield corresponding to $f(x)$ is 
isotopic to a Zhou-Pott semifield, ${\mathbb{S}}_{k,id}$.
\end{proof}

Theorems \ref{thm:Pott}, \ref{thm:BHfunc}, and \ref{thm:Pott2}, show that all of the planar functions described by Theorem \ref{thm:x2} 
have semifields that are isotopic to previously known semifields.





\section{Nuclei of Semifields \label{sec:nuclei}}
In this section we explore some results on the nuclei of semifields which can be constructed using the projection construction of Equation (\ref{eqn:newstar}).

\subsection{Middle Nuclei}

The following lemma is explicit but important.
\begin{lm}\label{r:equiv to x2}
Let $f(x)$ be a planar function over a finite field ${\mathbb{F}}$. If the middle nucleus of the commutative semifield associated with $f$ is ${\mathbb{F}}$, then the semifield associated with $f$ is ${\mathbb{F}}$, and hence,  $f(x)$ is EA-equivalent to $x^2\in {\mathbb{F}}[x]$.
\end{lm}

In order to determine the middle nuclei of semifields from the projection construction 
we will need the following proposition.

\begin{prop}\label{p:equiv to x2}
Suppose $f(x)=x^{p^n+1} + x^{(p^n+1)p^r} + x^{p^m+1} - x^{(p^m+1)p^r}$ is planar over ${\mathbb{F}}_{p^{2r}}$. If  $x^{p^m+1}$ is planar and $f(x)$ is EA-equivalent to $x^2$ over ${\mathbb{F}}_{p^{2r}}$, then $x^{p^m+1}=x^2$ and $n\equiv 0$ $($mod $r)$.
\end{prop}

\begin{proof}
Let $L_1(x)=c_0x+c_1x^p+c_2x^{p^2}+\cdot\cdot\cdot+c_{2r-1}x^{p^{2r-1}}$ and let  $L_2(x)=d_0x+d_1x^p+d_2x^{p^2}+\cdot\cdot\cdot+d_{2r-1}x^{p^{2r-1}}$ be permutation polynomials such that $L_1(x)^2=L_2(f(x))$.
Furthermore, let ${\mathbb{S}}_1$ be a commutative semifield associated with $f(x)$ and ${\mathbb{S}}_2$ be a commutative semifield associated with $f(x)|_{{\mathbb{F}}_{p^{r}}}$.  Since $f(x)$ is EA-equivalent to $x^2$ over ${\mathbb{F}}_{p^{2r}}$, the nucleus of ${\mathbb{S}}_1$ is ${\mathbb{S}}_1$. Thus, the nucleus of ${\mathbb{S}}_2$ is ${\mathbb{S}}_2$ and $f(x)|_{{\mathbb{F}}_{p^{r}}}$ is EA-equivalent to $x^2$ over ${\mathbb{F}}_{p^{r}}$.  Since $f(x)|_{{\mathbb{F}}_{p^{r}}}=2x^{p^n+1}$, $n\equiv 0$ $($mod $r)$.\\

\noindent\it Case 1. \rm Suppose $n\equiv 0$ $($mod $2r)$.
Note that
\begin{equation} (L_1(x))^2 = \sum_{j=0}^{2r-1}c_j^2x^{2p^j}
   + \sum_{0\leq j<\ell < 2r}\!2c_jc_{\ell}x^{p^j+p^{\ell}} \label{eqn:l1^2}
   \end{equation}
and \small
\begin{align*}
& L_2(f(x)) =\\
&  \sum_{j=0}^{r-1} (d_j+d_{j+r})\left(x^{2p^j}+x^{2p^{j+r}}\right) + (d_j-d_{j+r})\left(x^{(p^m+1)p^j}-x^{(p^m+1)p^{j+r}}\right).
\end{align*} \normalsize
Since $x^{p^m+1}$ is planar over ${\mathbb{F}}_{p^{2r}}$, $m \not\equiv r$ (mod $2r$). Therefore, $2c_jc_{j+r}=0$ for all $0\leq j <r$. So for every $0\leq j <r$ either $c_j=0$ or $c_{j+r}=0$. If $x^{p^m+1}\neq x^2$, then $c_j^2=(d_j+d_{j+r})=c_{j+r}^2$ for all $0\leq j <r$.  Hence, $c_j=c_{j+r}=0$ for any  $0\leq j <r$ and $L_1(x)=0$. By contradiction, $x^{p^m+1}= x^2$.\\

\noindent\it Case 2. \rm Suppose $n\equiv r$ $($mod $2r)$.
Then \small
\begin{align*}& L_2(f(x)) =\\
&  \sum_{j=0}^{r-1} 2(d_j+d_{j+r})x^{(p^r+1)p^j} + (d_j-d_{j+r})\left(x^{(p^m+1)p^j}-x^{(p^m+1)p^{j+r}}\right).
\end{align*} \normalsize
 If $x^{p^m+1}\neq x^2$, then $c_j^2=0$ for all $0\leq j <2r$ and $L_1(x)=0$. By contradiction, $x^{p^m+1}= x^2$.
\end{proof}
The following lemma will be also needed.
\begin{lm}\label{L:tedious}
Suppose $r$ and $\ell$ are integers such that $\ell|r$.  If $x$ is an element of a field, such that $x^{2\ell}=x$ and  2\,\large$\nmid$\normalsize\,$\frac{r}{\ell}$, then $x^{p^r}=x^{p^\ell}$.
\end{lm}

\begin{proof}
Since 2\,\large$\nmid$\normalsize\,$\frac{r}{\ell}$, $(2\ell)$\,\large$|$\normalsize\,$(r-\ell)$ and $x^{p^{r-\ell}}=x$. \mbox{Hence,
$x^{p^r}=\left(x^{p^{r-\ell}}\right)^{p^\ell}=x^{p^\ell}$}.
\end{proof}

\begin{prop}\label{p:middle nuc of f}
Suppose $f(x)=x^{p^n+1} + x^{(p^n+1)p^r} + x^{p^m+1} - x^{(p^m+1)p^r}$ is planar over ${\mathbb{F}}_{p^{2r}}$. Let $p^k$ be the order of the middle nucleus of ${\mathbb{S}}$, a commutative semifield associated with $f$.  If $x^{p^m+1}$ is planar over ${\mathbb{F}}_{p^{2r}}$, then gcd$(2n,m,2r)$ \large $|$ \normalsize $k$.
\end{prop}

\begin{proof}
The binary operation on the corresponding presemifield is
\begin{alignat*}{1}
 x*y 
     &= xy^{p^n}+x^{p^n}y+x^{p^r}y^{p^{n+r}}+x^{p^{n+r}}y^{p^r} \\
     &\quad+xy^{p^m}+x^{p^m}y-x^{p^r}y^{p^{m+r}}-x^{p^{m+r}}y^{p^r}.
\end{alignat*}
With $L(x)=1*x=2x+x^{p^n}+x^{p^m}+x^{p^{n+r}}-x^{p^{m+r}}$ the multiplication $\circ$ of the corresponding commutative semifield is defined by
\begin{alignat}{7}
&                        &(x*1) &\circ (1*y)\:\: &= &&x&*y \nonumber \\
&\Longleftrightarrow\:\: &L(x)  &\circ  L(y)     &= &&x&*y \nonumber \\
&\Longleftrightarrow     &x     &\circ  y        &= &&\:\:L^{-1}(x)&*L^{-1}(y)\,. \label{e:mult}
\end{alignat}
Let $t=$gcd$(2n,m,2r)$ and $a\in {\mathbb{F}}_{p^t}$. By definition, $L(a)$ is contained in the middle nucleus if and only if $(x\circ L(a))\circ y=x\circ(L(a)\circ y)$ for all $x,y\in{\mathbb{S}}$.  Since $L$ is a permutation and $\circ$ is commutative,  $L(a)$ is contained in the middle nucleus if and only if $(L(x)\circ L(a))\circ L(y)=(L(y)\circ L(a))\circ L(x)$ for all  $x,y\in{\mathbb S}$.  From Equation (\ref{e:mult}), this is equivalent to
\begin{equation}
L^{-1}(x*a)*y=L^{-1}(y*a)*x\,. \label{e:goal}
\end{equation}
Note that
\begin{alignat*}{1}
L(x)+L(x)^{p^r} 
                &= 2x+2x^{p^n}+2x^{p^r}+2x^{p^{n+r}} \\
\textnormal{and} \qquad L(x)-L(x)^{p^r} 
                &= 2x+2x^{p^m}-2x^{p^r}-2x^{p^{m+r}}.
\end{alignat*}
Thus, since $L(x)$ is additive,
\begin{alignat}{1}
x+x^{p^r} &= 2L^{-1}\!\left(x+x^{p^n}+x^{p^r}+x^{p^{n+r}}\right) \label{e:equa1} \\
\textnormal{and}\qquad x-x^{p^r} &= 2L^{-1}\!\left(x+x^{p^m}-x^{p^r}-x^{p^{m+r}}\right). \label{e:equa2}
\end{alignat}
Hence, from Equations (\ref{e:equa1}) and (\ref{e:equa2}) along with the facts that $a^{p^{2n}}=a$ and $a^{p^m}=a$ it follows that
\begin{alignat}{1}
L^{-1}(x*a) &= L^{-1}\!\left(xa^{p^n}+x^{p^n}a+x^{p^r}a^{p^{n+r}}+x^{p^{n+r}}a^{p^r}  \right. \nonumber \\
            &\quad\quad\quad\:\:\left.+xa^{p^m}+x^{p^m}a-x^{p^r}a^{p^{m+r}}-x^{p^{m+r}}a^{p^r}\right) \nonumber\\
            &= \frac{(xa^{p^n})+(xa^{p^{n}})^{p^r}}{2}+\frac{(xa)-(xa)^{p^r}}{2}. \label{eqn:15}
\end{alignat}

Since $x^{p^m+1}$ is planar over the field ${\mathbb{F}}_{p^{2r}}$, 2\,\large$\nmid$\normalsize\,$\frac{2r}{{\textnormal{gcd}}(m,2r)}$. Likewise, since $f(x)|_{{\mathbb{F}}_{p^r}}=2x^{p^n+1}$ is planar over ${\mathbb{F}}_{p^r}$ (but not necessarily planar over ${\mathbb{F}}_{p^{2r}}$), 2\,\large$\nmid$\normalsize\,$\frac{r}{{\textnormal{gcd}}(n,r)}$. In other words, the largest power of 2 dividing $2r$ also divides $m$ and the largest power of 2 dividing $r$ also divides $n$. Thus, either gcd$(n,m,2r)=$ gcd$(2n,m,2r)$ or the greatest power of 2 dividing $n$ equals the greatest power of 2 dividing $r$.  If gcd$(n,m,2r)=$ gcd$(2n,m,2r)$, then $a^{p^{{\textnormal{gcd}}(n,m,2r)}}\!=a$, and therefore, $a^{p^n}\!=a$. 
Otherwise, if the greatest power of 2 dividing $n$ equals the greatest power of 2 dividing $r$, then 2\,\large$\nmid$\normalsize\,$\frac{r}{{\textnormal{gcd}}(n,\frac{m}{2},r)}$ and, by Lemma~\ref{L:tedious}, $a^{p^n}\!=a^{p^{\textnormal{gcd}(n,m/2,r)}}\!=a^{p^r}$.
So either $a^{p^n}\!=a$ or $a^{p^n}\!=a^{p^r}$.\\

\noindent\it Case 1. \rm Suppose $a^{p^n}\!=a$.  Then from Equation~(\ref{eqn:15})
\begin{alignat*}{1}
L^{-1}(x*a) 
            &= xa
\end{alignat*}
and
\begin{equation*}
L^{-1}(x*a)*y = (xa)*y  = (ya)*x = L^{-1}(y*a)*x.
\end{equation*}

\noindent\it Case 2. \rm Suppose $a^{p^n}\!=a^{p^r}$.  Then from Equation~(\ref{eqn:15})
\begin{alignat*}{1}
L^{-1}(x*a) 
            &= \frac{x\!\left(a+a^{p^r}\right)+x^{p^r}\!\left(a-a^{p^{r}}\right)}{2}
\end{alignat*}
and
\begin{alignat}{3}
L^{-1}(x*a)*y &= \textstyle{\frac{x\left(a+a^{p^r}\right)+x^{p^r}\!\left(a-a^{p^{r}}\right)}{2}*y} && \nonumber \\
              &= \textstyle{\frac{y\left(a+a^{p^r}\right)+y^{p^r}\!\left(a-a^{p^{r}}\right)}{2}*x} &&  \label{e:case2}  \\
              &= L^{-1}(y*a)*x. &&   \nonumber
\end{alignat} \normalsize

Since Equation (\ref{e:goal}) is satisfied for any $x,y\in {\mathbb{F}}_{p^{2r}}$ and any $a\in {\mathbb{F}}_{p^t}$, the field  ${\mathbb{F}}_{p^t}$ is contained in the middle nucleus of the \mbox{commutative semifield associated with $f$}.
\end{proof}

\begin{thm}\label{T:mid nuclei}
Suppose $f(x)=x^{p^n+1} + x^{(p^n+1)p^r} + x^{p^m+1} - x^{(p^m+1)p^r}$ is planar over ${\mathbb{F}}_{p^{2r}}$.  If $x^{p^m+1}$ is planar over ${\mathbb{F}}_{p^{2r}}$, then the middle nucleus of a commutative semifield associated with $f$ has order equal to $p^{\textnormal{gcd}(2n,m,2r)}$.
\end{thm}

\begin{proof}
First note that $f(x)$ is a planar function over any subfield of ${\mathbb{F}}_{p^{2r}}$. Let $t=\textnormal{gcd}(2n,m,2r)$ and $p^k$ be the order of the middle nucleus.  By Proposition~\ref{p:middle nuc of f}, $t$\,\large$|$\normalsize\,$k$.   Assume that ${\mathbb{F}}_{p^{t}} \lneqq {\mathbb{F}}_{p^{k}} \leq {\mathbb{F}}_{p^{2r}}$. Note that since $x^{p^m+1}$ is planar over ${\mathbb{F}}_{p^{2r}}$, $2$\,\large$|$\normalsize\,$m$. Thus, $2$\,\large$|$\normalsize\,$t$, and therefore, $k=2\ell$ for some integer $\ell$, namely, $\ell={\textnormal{gcd}}(n,\frac{m}{2},r)$.

Since \mbox{$f(x)|_{{\mathbb{F}}_{p^r}}=2x^{p^n+1}$} is planar over ${\mathbb{F}}_{p^r}$, 2\,\large$\nmid$\normalsize\,$\frac{r}{{\textnormal{gcd}}(n,r)}$, and therefore, 2\,\large$\nmid$\normalsize\,$\frac{r}{\ell}$.  By Lemma~\ref{L:tedious}, if $b\in {\mathbb{F}}_{p^{2\ell}}$, then $b^r=b^\ell$.
Thus,
\begin{alignat*}{1}
f(x)|_{{\mathbb{F}}_{p^{2\ell}}} &= x^{p^n+1} + x^{(p^n+1)p^r} + x^{p^m+1} - x^{(p^m+1)p^r}  \\
                                 &= x^{p^n+1} + x^{(p^n+1)p^\ell} + x^{p^m+1} - x^{(p^m+1)p^\ell}\,.
\end{alignat*}

Since $f(x)$ is planar over ${\mathbb{F}}_{p^{2\ell}}$, the middle nucleus of the commutative semifield associated with $f(x)|_{{\mathbb{F}}_{p^{2\ell}}}$ is ${\mathbb{F}}_{p^k}={\mathbb{F}}_{p^{2\ell}}$.  By Lemma~\ref{r:equiv to x2}, $f(x)$ is EA-equivalent to $x^2$ over ${\mathbb{F}}_{p^{k}}$.
But, by Proposition~\ref{p:equiv to x2}, this is not possible unless $x^{p^m+1}=x^2$ over ${\mathbb{F}}_{p^{2\ell}}$ and $n\equiv 0$ $($mod $\ell)$.  Hence, $(2\ell)$\,\large$|$\normalsize\,$m$ and $\ell$\,\large$|$\normalsize\,$n$. In other words, $k$\,\large$|$\normalsize\,$m$ and $k$\,\large$|$\normalsize\,$(2n)$ \mbox{contradicting the assumption that ${\mathbb{F}}_{p^{t}} \lneqq {\mathbb{F}}_{p^{k}}$.}
\end{proof}


\subsection{Left Nuclei}
The following proposition can be used to calculate the left nuclei of some semifields construed using the projection construction


\begin{lm}\label{L:inverse} Suppose $f(x)=x^{p^n+1}+x^{(p^n+1)p^{3n}}+x^{p^{2n}+1}-x^{(p^{2n}+1)p^{3n}}$ is planar over ${\mathbb{F}}_{p^{6n}}$.  If $*$ is the binary operation of the corresponding presemifield and $L(x)=x*1=2x+x^{p^n}+x^{p^{2n}}+x^{p^{4n}}-x^{p^{5n}}$, then
\[ L^{-1}(x)={\textstyle{\frac{1}{4}}}x-{\textstyle{\frac{1}{4}}}x^{p^n}+{\textstyle{\frac{1}{4}}}x^{p^{5n}}. \]
\end{lm}

\begin{proof}
Let $F(x)=\frac{1}{4}x-\frac{1}{4}x^{p^n}+\frac{1}{4}x^{p^{5n}}\in {\mathbb{F}}_{p^{6n}}[x]$.  Then
\begin{alignat*}{9}
F(L(x))&= & &{\textstyle{\frac{1}{2}}}&x &+{\textstyle{\frac{1}{4}}}x^{p^n} &+{\textstyle{\frac{1}{4}}}x^{p^{2n}} &      &+{\textstyle{\frac{1}{4}}}x^{p^{4n}} &-{\textstyle{\frac{1}{4}}}x^{p^{5n}}\\
       &  &+&{\textstyle{\frac{1}{4}}}&x &-{\textstyle{\frac{1}{2}}}x^{p^n} &-{\textstyle{\frac{1}{4}}}x^{p^{2n}} &-{\textstyle{\frac{1}{4}}}x^{p^{3n}} &      &-{\textstyle{\frac{1}{4}}}x^{p^{5n}}\\
       &  &+&{\textstyle{\frac{1}{4}}}&x &+{\textstyle{\frac{1}{4}}}x^{p^n} &      &+{\textstyle{\frac{1}{4}}}x^{p^{3n}} &-{\textstyle{\frac{1}{4}}}x^{p^{4n}} &+{\textstyle{\frac{1}{2}}}x^{p^{5n}},\\
       &= &&&x&\,.&&&&
\end{alignat*}
\end{proof}
\begin{thm}\label{t:left}
Suppose $f(x)=x^{p^n+1}+x^{(p^n+1)p^{3n}}+x^{p^{2n}+1}-x^{(p^{2n}+1)p^{3n}}$ is planar over ${\mathbb{F}}_{p^{6n}}$. If ${\mathbb{S}}$ is its corresponding semifield, then   the left nucleus of ${\mathbb{S}}$ is equal to the middle nucleus and has order of $p^{2n}$.
\end{thm}

\begin{proof}
By Theorem~\ref{T:mid nuclei}, the middle nucleus of ${\mathbb{S}}$ is ${\mathbb{F}}_{p^{2n}}$.  Since the left nucleus of ${\mathbb{S}}$ is a subfield of the middle nucleus, it is enough to show that any element of ${\mathbb{F}}_{p^{2n}}$ is contained in the left nucleus. Let $a\in {\mathbb{F}}_{p^{2n}}$. Since $a^{p^n}=a^{p^{3n}}$, we are in case 2 of the proof of Proposition~\ref{p:middle nuc of f}.  Thus, by using Equation (\ref{e:case2}) with $m=2n$, $r=3n$ and $a^{p^{3n}}=a^{p^n}$ \small
\begin{alignat}{3}
L^{-1}(x*a)*y
      &=\quad \textstyle{\frac{x^{p^n}\!\!\left(a+a^{p^{3n}}\!\right)+x^{p^{4n}}\!\!\left(a^{p^{{3n}}}-a\right)}{2}y} \!\!\!
      &&+  \textstyle{\frac{x\left(a+a^{p^{3n}}\!\right)+x^{p^{3n}}\!\!\left(a-a^{p^{{3n}}}\!\right)}{2}y^{p^n}} \nonumber \\
      &\quad+  \textstyle{\frac{x^{p^{4n}}\!\!\left(a+a^{p^{3n}}\!\right)+x^{p^n}\!\!\left(a-a^{p^{{3n}}}\!\right)}{2}y^{p^{3n}}} \!\!\!
      &&+  \textstyle{\frac{x^{p^{3n}}\!\!\left(a+a^{p^{3n}}\!\right)+x\left(a^{p^{{3n}}}\!\!\!-a\right)}{2}y^{p^{4n}}} \nonumber \\
      &\quad+  \textstyle{\frac{x^{p^{2n}}\!\!\left(a+a^{p^{3n}}\!\right)+x^{p^{5n}}\!\!\left(a-a^{p^{{3n}}}\!\right)}{2}y} \!\!
      &&+  \textstyle{\frac{x\left(a+a^{p^{3n}}\!\right)+x^{p^{3n}}\!\!\left(a-a^{p^{{3n}}}\!\right)}{2}y^{p^{2n}}} \nonumber \\
      &\quad-  \textstyle{\frac{x^{p^{5n}}\!\!\left(a+a^{p^{3n}}\!\right)+x^{p^{2n}}\!\!\left(a^{p^{{3n}}}\!\!\!-a\right)}{2}y^{p^{3n}}} \!\!\!
      &&-  \textstyle{\frac{x^{p^{3n}}\!\!\left(a+a^{p^{3n}}\!\right)+x\left(a^{p^{{3n}}}\!\!\!-a\right)}{2}y^{p^{5n}}} \nonumber \\
      &=\quad \textstyle{\frac{x^{p^n}\!\!\left(a+a^{p^{n}}\!\right)+x^{p^{4n}}\!\!\left(a^{p^{{n}}}-a\right)}{2}y} \!\!\!
      &&+  \textstyle{\frac{x^{p^{2n}}\!\!\left(a+a^{p^{n}}\!\right)+x^{p^{5n}}\!\!\left(a-a^{p^{{n}}}\!\right)}{2}y} \nonumber \\
      &\quad+  \textstyle{\frac{x\left(a+a^{p^{n}}\!\right)+x^{p^{3n}}\!\!\left(a-a^{p^{{n}}}\!\right)}{2}y^{p^n}} \!\!\!
      &&+  \textstyle{\frac{x\!\left(a+a^{p^{n}}\!\right)+x^{p^{3n}}\!\!\left(a-a^{p^{{n}}}\!\right)}{2}y^{p^{2n}}} \nonumber \\
      &\quad+  \textstyle{\frac{x^{p^{4n}}\!\!\left(a+a^{p^{n}}\!\right)+x^{p^n}\!\!\left(a-a^{p^{{n}}}\!\right)}{2}y^{p^{3n}}} \!\!\!
      &&+  \textstyle{\frac{x^{p^{5n}}\!\!\left(-a-a^{p^{n}}\!\right)+x^{p^{2n}}\!\!\left(a-a^{p^{{n}}}\!\right)}{2}y^{p^{3n}}} \nonumber \\
      &\quad+  \textstyle{\frac{x^{p^{3n}}\!\!\left(a+a^{p^{n}}\!\right)+x\left(a^{p^{{n}}}\!\!\!-a\right)}{2}y^{p^{4n}}} \!\!\!
      &&+  \textstyle{\frac{x^{p^{3n}}\!\!\left(-a-a^{p^{n}}\!\right)+x\left(a-a^{p^{{n}}}\!\right)}{2}y^{p^{5n}}} \label{e:xay}
\end{alignat} \normalsize
for any $x,y\in{\mathbb{F}}_{p^{6n}}$.

By Lemma~\ref{L:inverse}, $L^{-1}(x)={\textstyle{\frac{1}{4}}}x-{\textstyle{\frac{1}{4}}}x^{p^n}+{\textstyle{\frac{1}{4}}}x^{p^{5n}}$.  Therefore, since $$\textstyle{x*y=xy^{p^n}\!\!+\!x^{p^n}y\!+\!x^{p^{3n}}y^{p^{4n}}\!\!+\!x^{p^{4n}}y^{p^{3n}}\!\!+\!xy^{p^{2n}}\!\!+\!x^{p^{2n}}y\!-\!x^{p^{3n}}y^{p^{5n}}\!\!-\!x^{p^{5n}}y^{p^{3n}}}\!\!,$$
\begin{alignat}{3}
L^{-1}(x*y)
      =\:\:&\textstyle{\left(\!x^{p^n}\!\!+\!x^{p^{2n}}\!\!+\!x^{p^{4n}}\!\!+\!x^{p^{5n}}\!\right)\frac{y}{4}} \,
      &&+  \!\textstyle{\left(\!x\!-\!x^{p^{2n}}\!\!-\!x^{p^{3n}}\!\!+\!x^{p^{5n}}\!\right)\frac{y^{p^n\!\!\!}}{\!\!\!4}} \nonumber \\
      +&\textstyle{\left(\!x\!-\!x^{p^{n}}\!\!+\!x^{p^{3n}}\!\!-\!x^{p^{4n}}\!\right)\frac{y^{p^{2n\!\!\!\!\!}}}{\!\!\!\!4}} \!\!\!
      &&+  \!\textstyle{\left(\!-x^{p^n}\!\!+\!x^{p^{2n}}\!\!+\!x^{p^{4n}}\!\!-\!x^{p^{5n}}\!\right)\frac{y^{p^{3n\!\!\!\!\!}}}{\!\!\!\!4}} \nonumber \\
      +&\textstyle{\left(\!x\!-\!x^{p^{2n}}\!\!+\!x^{p^{3n}}\!\!-\!x^{p^{5n}}\!\right)\frac{y^{p^{4n\!\!\!\!\!}}}{\!\!\!\!4}} \,
      &&+  \!\textstyle{\left(\!x\!+\!x^{p^{n}}\!\!-\!x^{p^{3n}}\!\!-\!x^{p^{4n}}\!\right)\frac{y^{p^{5n\!\!\!\!\!}}}{\!\!\!\!4}}\,. \label{e:Lxy}
\end{alignat}
Note that for any $z\in {\mathbb{F}}_{p^{6n}}$,
\begin{alignat}{1}
z*a&=\textstyle{za^{p^n}\!\!+\!z^{p^n}a\!+\!zz^{p^{3n}}a^{p^{4n}}\!\!+\!z^{p^{4n}}a^{p^{3n}}\!\!+\!za^{p^{2n}}\!\!+\!z^{p^{2n}}a\!-\!z^{p^{3n}}a^{p^{5n}}\!\!-\!z^{p^{5n}}a^{p^{3n}}} \nonumber \\
   &=\left(\!z\!+\!z^{p^n}\!\!+\!z^{p^{2n}}\!\!+\!z^{p^{3n}}\!\right)a +  \left(\!z\!-\!z^{p^{3n}}\!\!+\!z^{p^{4n}}\!\!-\!z^{p^{5n}}\!\right)a^{p^n}\!. \label{e:za}
\end{alignat}
By using Equality (\ref{e:Lxy}) and letting $z=L^{-1}(x*y)$,
\begin{alignat}{3}
z\!+\!z^{p^n}\!\!&+\!z^{p^{2n}}\!\!+\!z^{p^{3n}} = && \nonumber \\
      =\:\:&\textstyle{\left(\!2x^{p^n}\!\!+\!2x^{p^{2n}}\!\!-\!2x^{p^{4n}}\!\!+\!2x^{p^{5n}}\!\right)\frac{y}{4}} \,
      &&+  \!\textstyle{\left(\!2x\!+\!2x^{p^{3n}}\!\right)\frac{y^{p^n\!\!\!}}{\!\!\!4}} \nonumber \\
      +&\textstyle{\left(\!2x\!+\!x^{p^{3n}}\!\right)\frac{y^{p^{2n\!\!\!\!\!}}}{\!\!\!\!4}} \!\!\!
      &&+  \!\textstyle{\left(\!2x^{p^n}\!\!+\!2x^{p^{2n}}\!\!+\!2x^{p^{4n}}\!\!-\!2x^{p^{5n}}\!\right)\frac{y^{p^{3n\!\!\!\!\!}}}{\!\!\!\!4}} \nonumber \\
      +&\textstyle{\left(\!-2x\!+\!2x^{p^{3n}}\!\right)\frac{y^{p^{4n\!\!\!\!\!}}}{\!\!\!\!4}} \,
      &&+  \!\textstyle{\left(\!2x\!-\!2x^{p^{3n}}\!\right)\frac{y^{p^{5n\!\!\!\!\!}}}{\!\!\!\!4}} \label{e:zzzz}
\end{alignat}
and
\begin{alignat}{3}
z\!-\!z^{p^n}\!\!&+\!z^{p^{2n}}\!\!-\!z^{p^{3n}} =&& \nonumber \\
      =\:\:&\textstyle{\left(\!2x^{p^n}\!\!+\!2x^{p^{2n}}\!\!+\!2x^{p^{4n}}\!\!-\!2x^{p^{5n}}\!\right)\frac{y}{4}} \,
      &&+  \!\textstyle{\left(\!2x\!-\!2x^{p^{3n}}\!\right)\frac{y^{p^n\!\!\!}}{\!\!\!4}} \nonumber \\
      +&\textstyle{\left(\!2x\!-\!x^{p^{3n}}\!\right)\frac{y^{p^{2n\!\!\!\!\!}}}{\!\!\!\!4}} \!\!\!
      &&+  \!\textstyle{\left(\!-2x^{p^n}\!\!-\!2x^{p^{2n}}\!\!+\!2x^{p^{4n}}\!\!-\!2x^{p^{5n}}\!\right)\frac{y^{p^{3n\!\!\!\!\!}}}{\!\!\!\!4}} \nonumber \\
      +&\textstyle{\left(\!2x\!+\!2x^{p^{3n}}\!\right)\frac{y^{p^{4n\!\!\!\!\!}}}{\!\!\!\!4}} \,
      &&+  \!\textstyle{\left(\!-2x\!-\!2x^{p^{3n}}\!\right)\frac{y^{p^{5n\!\!\!\!\!}}}{\!\!\!\!4}}\,. \label{e:zzzz2}
\end{alignat}
Hence, from Equalities (\ref{e:za}), (\ref{e:zzzz}) and (\ref{e:zzzz2}), it follows that
\small
\begin{alignat}{3}
L^{-1}(x*y)*a
      &=\:\: \textstyle{\frac{x^{p^n}\!a+x^{p^{2n}}\!a-x^{p^{4n}}\!a+x^{p^{5n}}\!a}{2}y} \!\!\!
      &&+  \textstyle{\frac{x^{p^{n}}\!a^{p^{n}}\!+x^{p^{2n}}\!a^{p^{n}}\!+x^{p^{4n}}\!a^{p^n}\!-x^{p^{5n}}\!a^{p^{n}}}{2}y} \nonumber \\
      &\:\:+  \textstyle{\frac{xa+x^{p^{3n}}\!a+xa^{p^{n}}\!-x^{p^{3n}}\!a^{p^{{n}}}}{2}y^{p^n}} \!\!\!
      &&+  \textstyle{\frac{xa+x^{p^{3n}}\!a+xa^{p^{n}}\!-x^{p^{3n}}\!a^{p^{{n}}}}{2}y^{p^{2n}}} \nonumber \\
      &\:\:+  \textstyle{\frac{x^{p^{n}}\!a+x^{p^{2n}}\!a+x^{p^{4n}}\!a-x^{p^{5n}}\!a}{2}y^{p^{3n}}} \!\!\!
      &&+  \textstyle{\frac{-x^{p^{n}}\!a^{p^n}\!\!-x^{p^{2n}}\!a^{p^{n}}\!\!+x^{p^{4n}}\!a^{p^n}\!\!-x^{p^{5n}}\!a^{p^{{n}}}\!}{2}y^{p^{3n}}} \nonumber \\
      &\:\:+  \textstyle{\frac{-xa+x^{p^{3n}}\!a+xa^{p^{n}}\!+x^{p^{3n}}\!a^{p^{{n}}}}{2}y^{p^{4n}}} \!\!\!
      &&+  \textstyle{\frac{xa-x^{p^{3n}}\!a-xa^{p^{n}}\!-x^{p^{3n}}\!a^{p^{{n}}}}{2}y^{p^{5n}}}. \label{e:xya}
\end{alignat}
From Equations (\ref{e:xay}) and  (\ref{e:xya}), $L^{-1}(x*a)*y=L^{-1}(x*y)*a$ for any $x,y\in {\mathbb{F}}_{p^{6n}}$. Hence, $a$ is contained in the left nucleus of ${\mathbb{S}}$.
\end{proof}

\section{Additional semifields \label{sec:new}}


We now present several  planar polynomials obtained from \mbox{Equation~(\ref{eqn:newplanar})}.  By examining their corresponding semifields, we show that they are new.
With the aid of an algebra package such as GAP \cite{gap} it is possible to check if a function is planar.  Additionally, the following lemma 
reduces the amount of computations required.

\begin{lm}\label{lm:reduce}
Let $g(x)=x^{p^i+p^j}$ and $h(x)=x^{p^s+p^t}$ be DO monomials over ${\mathbb{F}}_{p^{2r}}$ and let $f(x)$ be defined as in \mbox{Equation (\ref{eqn:newplanar})}. Then for any $c\in {\mathbb{F}}_{p^r}^*$, $D_f(cx,ca)=0$ if and only if $D_f(x,a)=0$.
\end{lm}

\begin{proof}
For any $c\in {\mathbb{F}}_{p^r}^*$, \footnotesize
\begin{alignat*}{3}
 & & D_f(x,a)&=0 \\
 &\Longleftrightarrow & D_f(x,a)+D_f(x,a)^{p^r}=0 \:\textnormal{ and }\: D_f(x,a)-D_f(x,a)^{p^r} &=0 \\
 &\Longleftrightarrow & 2x^{p^i}a^{p^j} + 2x^{p^j}a^{p^i}+2x^{p^{i+r}}a^{p^{j+r}} + 2x^{p^{j+r}}a^{p^{i+r}}&=0 \\
 && \textnormal{ and }\: 2x^{p^s}a^{p^t} + 2x^{p^t}a^{p^s}-2x^{p^{s+r}}a^{p^{t+r}} - 2x^{p^{t+r}}a^{p^{s+r}}&=0 \\
 &\Longleftrightarrow & \frac{1}{c^{p^{i}+p^{j}}}\left(2(cx)^{p^i}\!\!(ca)^{p^j} \!\!+ 2(cx)^{p^j}\!\!(ca)^{p^i} \!\!+ 2(cx)^{p^{i+r}}\!\!(ca)^{p^{j+r}} \!\!+ 2(cx)^{p^{j+r}}\!\!(ca)^{p^{i+r}}\right)&=0 \\
 && \textnormal{and }\: \frac{1}{c^{p^{s}+p^{t}}}\left(2(cx)^{p^s}\!\!(ca)^{p^t} \!\!+ 2(cx)^{p^t}\!\!(ca)^{p^s} \!\!-2(cx)^{p^{s+r}}\!\!(ca)^{p^{t+r}} \!\!- 2(cx)^{p^{t+r}}\!\!(ca)^{p^{s+r}}\right)&=0 \\
 &\Longleftrightarrow & D_f(cx,ca)+D_f(cx,ca)^{p^r}=0 \:\textnormal{ and }\: D_f(cx,ca)-D_f(cx,ca)^{p^r} &=0 \\
 &\Longleftrightarrow & D_f(cx,ca)&=0.
\end{alignat*} \normalsize
\end{proof}
From Lemma~\ref{lm:reduce}, one only has to check one value of $x$ for each of the $p^r+1$ cosets of ${\mathbb{F}}_{p^r}^*$ in ${\mathbb{F}}_{p^{2r}}^*$. Let $\alpha$ be a primitive element of ${\mathbb{F}}_{p^{2r}}$, then by using Lemma~\ref{lm:reduce} it can be determined  if \[ f(x)=x^{p^i+p^j} + x^{(p^i+p^j)p^r} + x^{p^s+p^t} - x^{(p^s+p^t)p^r} \]
is planar over ${\mathbb{F}}_{p^{2r}}$ just by checking
if $\alpha^k$ is a root of $D_f(x,a)\in {\mathbb{F}}_{p^{2r}}[x]$ for $0\leq k\leq p^r$ (or $1\leq k\leq p^r+1$) and $a\in {\mathbb{F}}_{p^{2r}}^*$.   Thus, using  Lemma~\ref{lm:reduce} along with GAP \cite{gap}, one can easily verify that the following are examples of planar functions of the form of \mbox{Equation (\ref{eqn:newplanar})}.

\begin{ex}\label{e:exam}
The following are examples of planar functions of the form
\[  x^{p^i+1} + x^{(p^i+1)p^r} + x^{p^s+1} - x^{(p^s+1)p^r} \in {\mathbb{F}}_{p^{2r}}[x]. \]
\begin{enumerate}
\item $x^{6} + x^{(6)5^3} + x^{26} - x^{(26)5^3} \in {\mathbb{F}}_{5^{6}}[x]$
\item $x^{8} + x^{(8)7^3} + x^{50} - x^{(50)7^3} \in {\mathbb{F}}_{7^{6}}[x]$
\item $x^{10} + x^{(10)3^6} + x^{82} - x^{(82)3^6} \in {\mathbb{F}}_{3^{12}}[x]$
\item $x^{10} + x^{(10)3^6} + x^{28} - x^{(28)3^6} \in {\mathbb{F}}_{3^{12}}[x]$
\end{enumerate}
\end{ex}

Note that in the first three functions in Example~\ref{e:exam}, the function is of the form of \mbox{Equation (\ref{eqn:newplanar})} where $h(x)$ is a planar monomial.  Whereas, $h(x)$ is not planar in the fourth function in Example~\ref{e:exam}.  Further note that the first three functions are of the form $f(x)=x^{p^n+1}+x^{(p^n+1)p^{3n}}+x^{p^{2n}+1}-x^{(p^{2n}+1)p^{3n}}\in {\mathbb{F}}_{p^{2r}}[x]$ where $r=3n$ in which case Theorem  \ref{T:mid nuclei} may be used to determine the sizes of the middle nuclei of the corresponding semifields.

\begin{cor}\label{c:middlenuc}
The orders of the middle nuclei of the semifields corresponding to the first three functions in Example~\ref{e:exam} are 25, 49 and 81 respectively.
\end{cor}




The nuclei of semifield corresponding to the  fourth function of Example \ref{e:exam} may also be calculated.
\begin{thm}\label{lem:9}
If \,${\mathbb{S}}$ is a corresponding semifield of the planar function $f(x)=x^{10} + x^{(10)3^6} + x^{28} - x^{(28)3^6}\in {\mathbb{F}}_{3^{12}}[x]$, then  the middle nucleus of ${\mathbb{S}}$ is ${\mathbb{F}}_{9}$ and the left nucleus of ${\mathbb{S}}$ is ${\mathbb{F}}_{3}$.
\end{thm}

\begin{proof}
It is easy to check, by direct computation using GAP \cite{gap}, that neither ${\mathbb{F}}_{27}$ nor ${\mathbb{F}}_{81}$ lies in the middle nucleus of ${\mathbb{S}}$. Furthermore, by using GAP and letting $a$ be a generator of ${\mathbb{F}}_{9}$, it is easy to show that $a$, and therefore the field  ${\mathbb{F}}_{9}$, lies in the middle nucleus of ${\mathbb{S}}$.

Since the left nucleus of ${\mathbb{S}}$ is a subfield of the middle nucleus, the left nucleus of ${\mathbb{S}}$ is either ${\mathbb{F}}_{3}$ or ${\mathbb{F}}_{9}$. By direct computation, ${\mathbb{F}}_{9}$ does not lie in the left nucleus of ${\mathbb{S}}$.  Hence, the left nucleus of ${\mathbb{S}}$ is ${\mathbb{F}}_{3}$.
\end{proof}


By Corollary~\ref{c:middlenuc} and  Lemma~\ref{lem:9} along with \cite[\S 3]{Marino} and \cite{ZhouPott2013}, the semifields associated with the functions of Example~\ref{e:exam} are either new or otherwise isotopic to either the Budaghyan-Helleseth $BH(p,r,s,\beta)$ or the Zhou-Pott ${\mathbb{S}}_{k,\sigma}$ semifields.
We now proceed to show that the semifields corresponding to the 
functions in Example~\ref{e:exam} are not  isotopic to either the Budaghyan-Helleseth  or the Zhou-Pott semifields.  

\begin{thm}
The semifields corresponding to the first three functions in Example~\ref{e:exam} are not isotopic to either the Budaghyan-Helleseth $BH(p,r,s,\beta)$ or the Zhou-Pott  ${\mathbb{S}}_{k,\sigma}$ semifields and are therefore not isotopic to any previously known semifields.
\end{thm}

\begin{proof}
It is known from \cite[Thm 4.1]{Marino} that the Budaghyan-Helleseth presemifields $BH(p,r,s,\beta)$ have middle nuclei of order $p^{2\textnormal{gcd}(r,s)}$ and left nuclei of order $p^{\textnormal{gcd}(r,s)}$. Since their left nuclei are proper subfields of their middle nuclei, by Theorem~\ref{t:left}, the semifields of the first three functions in Example~\ref{e:exam} are not isotopic to the Budaghyan-Helleseth semifields $BH(p,r,s,\beta)$.

By Theorem 2 of \cite{ZhouPott2013}, the order of the middle nucleus of a Zhou-Pott semifield ${\mathbb{S}}_{k,\sigma}$ is $p^{{2\textnormal{gcd}}(r,k)}$ when $\sigma=1$ and $p^{{\textnormal{gcd}}(r,k)}$ when $\sigma\neq1$.  But by Theorem~\ref{T:mid nuclei},  the semifields of the first three functions in Example~\ref{e:exam} have middle nuclei of order $p^{\textnormal{gcd}(2n,m,2r)}=p^{\textnormal{gcd}(2n,2n,6n)}=p^{2n}$.  Since 2\,\large$\nmid$\normalsize\,$\frac{r}{\textnormal{gcd}(n,r)}$, the largest power of 2 dividing $2r$ divides $2n$ but does not divide ${\textnormal{gcd}}(r,k)$.  Therefore, the semifields of the first three functions in Example~\ref{e:exam} are not isotopic to a Zhou-Pott semifield ${\mathbb{S}}_{k,\sigma}$ unless $\sigma=1$.  But if $\sigma=1$, then, by Theorem~3 of \cite{ZhouPott2013}, the order of the left nucleus of the Zhou-Pott semifield ${\mathbb{S}}_{k,\sigma}$ is $p^{\textnormal{gcd}(r,k,0)}=p^{\textnormal{gcd}(r,k)}$.  Hence, if $\sigma=1$, then the left nucleus of ${\mathbb{S}}_{k,\sigma}$ is a proper subfield of the middle nucleus.   Therefore, the semifields of the first three functions in Example~\ref{e:exam} are not isotopic to the Zhou-Pott \mbox{semifields ${\mathbb{S}}_{k,\sigma}$}.
\end{proof}





\begin{thm}
The planar function $f(x)=x^{10} + x^{(10)3^6} + x^{28} - x^{(28)3^6}\in {\mathbb{F}}_{3^{12}}[x]$ is not EA-equivalent to any previously known planar functions and its corresponding semifield is not 
 isotopic to any previously known semifields.
\end{thm}

\begin{proof}
Suppose $f(x)$ is EA-equivalent to a planar function, $\varphi(x)$, associated with a Zhou-Pott presemifield ${\mathbb{P}}_{k,\sigma}=({\mathbb{F}}_{p^{2r}},+,*)$ where $p=3$ and $r=6$.  Let $\omega \in {\mathbb{F}}_{p^{2r}}^*$ such that $\omega^{p^r}=-\omega$ and denote any $x\in {\mathbb{F}}_{p^{2r}}$ by $x=x_1+x_2\omega$ for $x_1,x_2\in {\mathbb{F}}_{p^{r}}$.
By \cite[Thm 1]{ZhouPott2013},
\begin{alignat*}{1}
\varphi(x) &= x*x = (x_1+x_2\omega)*(x_1+x_2\omega), \\
           &= 2x_1^{p^k+1}+\alpha\left(2x_2^{p^k+1}\right)^{\sigma}+2x_1x_2\omega, \\
           &= 2x_1^{p^k+1}+2\alpha\left(x_2^{p^k+1}\right)^{\sigma}+\textstyle{\frac{1}{2}}\left[(x_1+x_2\omega)^2-(x_1+x_2\omega)^{2p^r}\right], \\
           &= 2x_1^{p^k+1}+2\alpha\left(x_2^{p^k+1}\right)^{\sigma}+\textstyle{\frac{1}{2}}x^2-\textstyle{\frac{1}{2}}x^{2p^r},
\end{alignat*}
and therefore, $\varphi(x)-\varphi(x)^{p^r}=x^2-x^{2p^r}$.

Now let $L_1(x)=c_0x+c_1x^p+c_2x^{p^2}+\cdot\cdot\cdot+c_{2r-1}x^{p^{2r-1}}\!$ and $L_2(x)$ be linear permutation polynomials such that $\varphi(L_1(x))=L_2(f(x))$.  Note that \small
\begin{alignat*}{1}
\varphi(L_1(x))&-\varphi(L_1(x))^{p^r} = L_1(x)^2-L_1(x)^{2p^r} \\
          &= \sum_{0\leq j< r}(c_j^2-c_{j+r}^{2p^r})x^{2p^j} + (c_{j+r}^2-c_{j}^{2p^r})x^{2p^{j+r}} \\
          &\quad+ \sum_{0\leq j<\ell < r}\!(2c_jc_{\ell}-2c_{j+r}^{p^r}c_{\ell+r}^{p^r})x^{p^j+p^{\ell}} + (2c_{j+r}c_{\ell+r}-2c_{j}^{p^r}c_{\ell}^{p^r})x^{p^{j+r}+p^{\ell+r}} \\
          &\quad+ \sum_{\begin{array}{c}0\leq j< r\\r\leq \ell<2r\end{array}}\!(2c_jc_{\ell}-2c_{j+r}^{p^r}c_{\ell-r}^{p^r})x^{p^j+p^{\ell}}.
\end{alignat*} \normalsize
Since $\varphi(L_1(x))-\varphi(L_1(x))^{p^r}=L_2(f(x))-L_2(f(x))^{p^r}$, $c_j^2-c_{j+r}^{2p^r}=0$ and $2c_jc_{j+1}-2c_{j+r}^{p^r}c_{j+1+r}^{p^r}=0$ for any $0\leq j< r$.  Thus, $c_{j+r}^{p^r}=\pm c_j$ and $c_{j+r}^{p^r}c_{j+1+r}^{p^r}=c_jc_{j+1}$ for all $0\leq j< r$.  So if $c_{j+r}^{p^r}= c_j$, then $c_{j+1+r}^{p^r}= c_{j+1}$ and, likewise, if $c_{j+r}^{p^r}= -c_j$, then $c_{j+1+r}^{p^r}= -c_{j+1}$. Hence, there exists an $a\in \{\pm 1\}$ such that $c_{j+r}^{p^r}=a c_j$ for all $0\leq j< r$.  Therefore, $L_1(x)^2-L_1(x)^{2p^r}=0$ meaning $L_1(x)^2\in {\mathbb{F}}_{p^r}$ for any $x\in {\mathbb{F}}_{p^{2r}}$.  This contradicts the fact that $L_1(x)\in {\mathbb{F}}_{p^{2r}}[x]$ is a permutation polynomial. Hence, $f(x)$ is not EA-equivalent to a planar function, $\varphi(x)$, corresponding to a Zhou-Pott presemifield.

Now assume that $f(x)$ is EA-equivalent to the planar function, $$B(x)=x^{p^r+1}+\omega\beta x^{p^s+1}+\omega\beta^{p^r}x^{(p^s+1)p^r},$$  where $p=3$ and $r=6$.
Let $L_1(x)=c_0x+c_1x^p+c_2x^{p^2}+\cdot\cdot\cdot+c_{2r-1}x^{p^{2r-1}}$ and $L_2(x)$ be linear permutation polynomials such that $B(L_1(x))=L_2(f(x))$.  Since $B(x)+B(x)^{p^r}=2x^{p^r+1}$, \small
\begin{alignat*}{1}
B(L_1(x))&+B(L_1(x))^{p^r} = 2L_1(x)^{p^r+1} \\
          &= 2\sum_{0\leq j< r}\left(c_jc_{j+r}x^{2p^j} + c_{j+r}c_{j}^{p^r}\right)x^{2p^{j+r}} \\
          &\quad+ 2\sum_{0\leq j<\ell < r}\!(c_jc_{\ell+r}^{p^r}+c_{j+r}^{p^r}c_{\ell})x^{p^j+p^{\ell}} + (c_{j+r}c_{\ell}^{p^r}+c_{j}^{p^r}c_{\ell+r})x^{p^{j+r}+p^{\ell+r}} \\
          &\quad+ 2\sum_{\begin{array}{c}0\leq j< r\\r\leq \ell<2r\end{array}}\!(c_jc_{\ell-r}^{p^r}+c_{j+r}^{p^r}c_{\ell})x^{p^j+p^{\ell}}.
\end{alignat*} \normalsize
Since $B(L_1(x))+B(L_1(x))^{p^r}=L_2(f(x))+L_2(f(x))^{p^r}$, the coefficient of $x^{2p^j}$, namely $c_jc_{j+r}^{p^r}$, is zero for any $0\leq j< r$. Likewise, the coefficient of $x^{p^j+p^{j+r}}$, namely $c_j^{p^r+1}+c_{j+r}^{p^r+1}$, is zero for all $0\leq j< r$.  Thus, $c_j=0=c_{j+r}$ for all $0\leq j< r$, meaning $L_1(x)=0$.  This contradicts the fact that $L_1(x)$ is a permutation polynomial. Hence, $f(x)$ is not EA-equivalent to the generalized Budaghyan-Helleseth planar function.
\end{proof}

\section*{Acknowledgements}
Thanks to the anonymous referees who's suggestions have been valuable in improving this paper.





\end{document}